\documentclass{amsart}
\usepackage{amsmath}
\usepackage[T1]{fontenc}
\usepackage{textcomp} 
\DeclareEncodingSubset{TS1}{hlh}{1}  
\usepackage{mathalfa} 

\usepackage{amscd,amssymb,epsfig, epsf,pinlabel,graphicx}
\usepackage{epic,eepic}
\usepackage[dvipsnames]{xcolor}

\usepackage[all]{xy}
\SelectTips{cm}{}
\allowdisplaybreaks

\numberwithin{equation}{subsection}

\newtheorem{propo}{Proposition}[section]
\newtheorem{corol}[propo]{Corollary}
\newtheorem{theor}[propo]{Theorem}
\newtheorem{lemma}[propo]{Lemma}

\theoremstyle{definition}

\theoremstyle{remark}

\let\oldmarginpar\marginpar
\renewcommand\marginpar[1]{\oldmarginpar{\footnotesize #1}}

\newcommand{\card}{\operatorname{card}}

\newcommand{\id}{\operatorname{id}}

\begin{document}

\title[Strings in metric spaces]{Strings in metric spaces}

    \author[Vladimir Turaev]{Vladimir Turaev}
    \address{Vladimir Turaev \newline
    \indent   Department of Mathematics \newline
    \indent  Indiana University \newline
    \indent Bloomington IN47405, USA\newline
    \indent $\mathtt{vturaev@yahoo.com}$}

\begin{abstract}  We introduce
strings in metric spaces and define string complexes of  metric spaces. We describe the class of 2-dimensional topological spaces which arise in this way from finite metric spaces. \end{abstract}

\maketitle

\section {Introduction}

 Classical constructions due to \u{C}ech and Vietoris-Rips  associate certain simplicial complexes with metric spaces. The study of these complexes 
yields various invariants of  metric spaces including their \u{C}ech and Vietoris-Rips homology. This area of research has been  active recently in the context of topological data analysis, see \cite{JJ}, \cite{RB}.

We  discuss a different method deriving a simplicial complex from a metric space.  This method is based on the study of  points  for which the  triangle inequality is an equality.  Our construction is formulated in terms of so-called strings which we view as metric analogues of straight segments. We call the resulting simplicial complexes  the string  complexes. Homology of the string complex may be a useful invariant of the original metric space.

 Our main result yields necessary and sufficient conditions on a 2-dimensional topological space to be  homeomorphic to the string complex of a finite metric space. In particular, we show that all compact surfaces arise in this way.
 
 For the sake of generality, we define and  study strings in so-called gap spaces which generalize metric spaces by dropping all conditions on the metric except the triangle inequality.


  \section{Strings in gap spaces}\label{strings}

 \subsection{Gap spaces}\label{topophlsett} Given a set~$X$, we call a mapping $d:X\times X\to \Bbb R$ a \emph{gap function} if  it obeys the triangle inequality: $d(x,y)+d(y,z)\geq d(x,z)$ for all $x,y,z\in X$.  Such a pair $(X,d)$ is    a \emph{gap space}, and the number $d(x,y)$ is  the \emph{gap} between $x, y \in X$. 
   A repeated application of the triangle inequality  shows that for any  $n\geq 3$ points $x_1,x_2, ..., x_n\in X$   we have  the \emph{$n$-gon inequality}
   \begin{equation}\label{wo3}  d(x_1, x_{2})+d(x_2, x_{3})+\cdots+d(x_{n-1}, x_{n})\geq d(x_1, x_n).
     \end{equation} 
     
     A standard source of gap functions are road maps with some of the roads one-way and other roads two-way.  Here~$X$ is the set of all crossings of the roads  and the gap between any two crossings is the length of the shortest path leading from the first crossing to the second. We have to require that the map allows to travel from any crossing to any other crossing along the roads.

\subsection{Strings}\label{topolsett}  
A  \emph{string of  length $n\geq 3$}  in a gap space $(X,d)$ is an $n$-element  set $S \subset X$ whose elements   can be ordered  $x_1, x_2, \ldots, x_{n}$ so that  \eqref{wo3} is an equality:
 \begin{equation}\label{wo35} d(x_1, x_{2})+d(x_2, x_{3})+\cdots+d(x_{n-1}, x_{n}) =d(x_1, x_{n}).\end{equation}
Any order  in~$S$ satisfying  \eqref{wo35} is said to be \emph{direct}. A string endowed with a direct order is  \emph{ordered}.

It is convenient to extend these notions to   1-element and 2-element subsets of the gap space~$X$. For $x\in X$,  a 
 $1$-element set $\{x\} $ is  a \emph{string of  length~$1$}  
if~$x$ belongs to  a string of length~3 in~$X$. The only order in the set $\{x\} $ is direct.
Similarly,  a  $2$-element subset of~$ X$  is a \emph{string of  length~$2$}  
if it is contained (as a subset) in a  string of length~3 in~$X$.   
An order in a string of length~2 is  \emph{direct} if it extends to a direct order in a string of length 3. (We agree that the order $x,y,z$ in the 3-element set $\{x,y,z\}$ is an extension of 
 the order $x,y$ in the set $\{x,y\}$,  the order $x,z$ in the  set $\{x,z\}$, and  the order $y,z$ in the set $\{y,z\}$.)

\begin{lemma}\label{le+1}  Each  non-void subset of a string    is itself a string (called a substring). A direct order in a string restricts to a direct order in each substring. \end{lemma}
  
\begin{proof}  Let  $(X,d)$ be a gap space. Let $S\subset X$ be a string of length~$n$ with direct order  $S=\{x_1,..., x_n\}$.   For $n=1,2,3$,  both claims of the lemma concerning the subsets of~$S$ follow directly from the definitions.   Assume   that $n\geq 4$. 
 We have  \begin{equation}\label{wo309}d(x_1, x_{2})+d(x_2, x_{3})+\cdots+d(x_{n-2}, x_{n-1})+d(x_{n-1}, x_{n})
  \end{equation}
  $$\geq 
 d(x_1, x_{n-1})+d(x_{n-1}, x_{n})\geq d(x_1, x_n)$$
 where we use the $(n-1)$-gon inequality and  the triangle inequality.  Since $x_1,..., x_n$ is a direct order in~$S$, Formula   \eqref{wo35} implies that both inequalities in \eqref{wo309}  are equalities. The first of these equalities implies that the set  $\{x_1,..., x_{n-1}\}$  is a string with direct order  $x_1,..., x_{n-1}$.      Similarly,  we have $$d(x_1, x_{2})+d(x_2, x_{3})+\cdots+d(x_{n-1}, x_{n}) \geq 
 d(x_1, x_{2})+d(x_{2}, x_{n})\geq d(x_1, x_n). $$
 Again,  Formula \eqref{wo35} implies that both these inequalities are equalities. The first of them means that the  set $\{x_2,..., x_n\}$ is  a string with direct order  $x_2,..., x_n$.     Now, pick any $j\in \{2,..., n-1\}$. We have $$d(x_1, x_{2})+d(x_2, x_{3})+\cdots +d(x_{n-1}, x_{n})$$
 $$ = \sum_{k=1}^{j-2} d(x_k, x_{k+1})+d(x_{j-1}, x_{j})
 +d( x_{j}, x_{j+1})+\sum_{k=j+1}^{n-1} d( x_{k}, x_{k+1}) \geq 
 $$
 $$\geq   \sum_{k=1}^{j-2} d(x_k, x_{k+1})+d(x_{j-1},  x_{j+1}) +\sum_{k=j+1}^{n-1} d( x_{k}, x_{k+1} )
\geq d(x_1, x_n) $$
 where we  use the triangle inequality and   the  $(n-1)$-gon inequality.
 Since $x_1,..., x_n$ is a direct order in~$S$, both inequalities here have to be equalities. The second  of them  shows that the set   $\{x_1,...,x_{j-1}, x_{j+1},..., x_n\}$ is a string with direct order  $x_1,...,x_{j-1}, x_{j+1},..., x_n$.    This proves
 both claims of the lemma for all subsets  of~$S$ having  $n-1$ elements. Proceeding by induction,
 we obtain both claims for all non-void subsets of~$S$. 
\end{proof} 

The gaps between the points of an ordered string    $x_1,..., x_n$ of length $n\geq 3$ can be computed by the formula \begin{equation}\label{wo116} d(x_i, x_j)=d(x_i, x_{i+1})+d(x_{i+1}, x_{i+2})+\cdots+d(x_{j-1}, x_{j})\end{equation} for any  $1\leq i<j\leq n$.
 Indeed, by Lemma~\ref{le+1}, the sequence $x_i, x_{i+1},..., x_j$
is  an ordered string. Hence, we have \eqref{wo116}.


\subsection{The case of metric spaces}  More can be said about  strings in  metric spaces.
 A gap space  $(X, d:X\times X\to \Bbb R)$ is   a metric space (and~$d$ is a metric) iff $d(x,y)=d(y,x)\geq 0$   for all $x,y\in X$ and   $d(x,y)=0 \Longleftrightarrow x=y$.   
If $x_1,  \ldots, x_{n}$
is a direct order in a string in a metric space~$X$  with $n\geq 2$ then the opposite order $x_n, x_{n-1},..., x_1$
in the same string  also is direct.   Formula~\eqref{wo116} implies that  $x_1, x_n$ are the points of this string with the greatest distance between them. Therefore the unordered pair of points $x_1, x_n$ depends only on the string and not on the direct order in it. The points $x_1, x_n$  are called the \emph{endpoints} of the string.   Using these observations  and the induction on~$n$, one easily concludes that   in a metric space,  every string   of length $\geq 2$ has exactly two direct orders, and they are opposite to each other.

Strings in  metric spaces are unstable in the sense that a slight deformation of the metric can destroy    strings. For example, increasing all distances between  different points by the same positive number we obtain a new metric having no strings. On the other hand, every finite metric space can be \emph{trimmed} (see \cite{Tu1}, \cite{Tu2}) to produce a metric space $(X,d)$ with the following  property: for any $y\in X$ there are $x,z\in X\setminus \{y\}$ such that $x\neq  z$ and  $d(x,y)+d(y,z)=d(x,z)$. 
 Consequently,  all points of~$X$ lie on strings.

  \section{The string space}\label{stringcom}

\subsection{Basics}\label{com} Recall that an abstract simplicial complex~$C$ is a collection of non-empty finite subsets of a certain set such that for every element of the collection~$C$, all its non-empty subsets  also belong to~$C$. By Lemma~\ref{le+1},  the collection of strings in any gap space $(X,d)$ is an abstract simplicial complex. We  denote its geometric realization  by $\vert X,d \vert$.  
This is a topological space glued from simplexes numerated by the strings in~$X$. The simplex numerated by a string of length $n\geq 1$ has~$n$ vertexes and  dimension $n-1$.
The simplexes corresponding to smaller strings are faces of  the simplexes corresponding to bigger strings. We call $\vert X,d\vert$ the \emph{string space} of  $(X,d)$.
 By the very definition, the space  $\vert X,d\vert$ 
 admits a triangulation whose all vertexes (0-simplexes) and  edges (1-simplexes) are faces of certain 2-simplexes. The usual techniques of 
 algebraic topology - homology, cohomology, homotopy groups etc.\ -  apply to the  space  $\vert X,d\vert$ 
  and produce  interesting algebraic objects
associated with the gap space  $(X,d)$. In particular, all these constructions apply to metric spaces. 



\subsection{Examples}\label{exa} 1. Let a metric~$d$ in a set~$X$ be defined by  $d(x,y)=1$ for all $x \neq y$ and $d(x,x)=0$ for all $x\in X$. Then  $(X,d)$ has no strings and  $\vert X,d\vert=\emptyset$.

2. Let $X\subset \Bbb R$ be a finite set of real numbers with  metric induced by the standard metric in~$\Bbb R$. Suppose that $\card (X)\geq 3$. Then all non-void subsets of~$X$ are   strings. The string space  of~$X$ is a  simplex of dimension  $\card (X)-1$.

 
 
 3. Let $X\subset \Bbb R^2$ be a set consisting of $m\geq 3$ points lying on a straight line in $\Bbb R^2$  and $n\geq 3$ points lying 
on a parallel line. The metric in~$X$ is induced by the Euclidean metric  in $\Bbb R^2$. Then the string space of~$X$ is the disjoint union of two simplexes, one of dimension $m-1$ and one of dimension $n-1$. 

4. Consider  a convex polygon $P\subset \Bbb R^2$ with $k\geq 3$ vertexes and all interior angles  $<180^\circ$.  Consider a finite set $X\subset P$ including all vertexes of~$P$ and at least one point  inside each edge of~$P$.  The metric in~$X$ is the usual Euclidean metric. Then the string space $\vert X\vert$ is formed by~$k$ simplexes  $T_1,..., T_k$. If $\{e_i \,\vert \, i \in \Bbb Z/k\Bbb Z\}$ are the cyclically numerated  edges of~$P$ then  $\dim(T_i)=\card (X\cap e_i)-1\geq 2$. For all~$i$, the simplexes $T_i, T_{i+1}$ meet in one common vertex; otherwise,  the simplexes $T_1,..., T_k$ do not meet. Note  that the string space $\vert X\vert$ is homotopy equivalent to the circle.


5. Let $X\subset \Bbb R^2$ be a 4-element set formed by the vertexes of a planar rectangle. To define a metric in~$X$, we inscribe~$X$ in a circle in $\Bbb R^2$ and assign to any two points of~$X$  the length of the shortest arc which they bound in this circle. It  is easy to see that all non-void subsets of~$X$  are strings except~$X$ itself. The string space  of~$X$ is the boundary of a 3-simplex, i.e.,  a topological (and piecewise linear) 2-sphere.


\subsection{Remark} For any pair of distinct points $x,y$ of a metric space $(X,d)$, we can consider  a subcomplex of  the string complex $\vert X,d\vert$ formed by  simplexes corresponding to  strings  with endpoints $x,y$. Homology of this subcomplex yields  interesting algebraic data
associated with the metric space  $(X,d)$  and its points $x,y$. 
 
 \section{Two-dimensional string spaces}\label{dim2} 

 What topological spaces  arise as the string spaces of  finite metric spaces? By  definition, such a topological space   has a finite triangulation whose all vertexes and edges are vertexes and edges of   2-simplexes. The author does not know whether this condition is sufficient (conjecturally, not). The next theorem - our main result - establishes sufficiency of this condition  for 2-dimensional spaces.

 \begin{theor}\label{thm1}  Let $T$ be  a topological space  which has a finite 2-dimensional  triangulation whose every vertex is a vertex of a 2-simplex and  every edge is an edge of a 2-simplex. Then there exists a finite metric space  whose string space 
  is homeomorphic to~$T$. \end{theor}
  
\begin{proof}  We first pick any real numbers $k>0$ and  $u,v\in (k/2, k)$.  Pick a triangulation~$\tau$ of~$T$ as in the assumptions of the theorem. Let~$X$ be the  set of simplexes of~$\tau$.   Clearly, there is a unique function $d:X\times X\to \Bbb R$ such that:

 $-$ \, $d$ is symmetric, i.e., $d(x,y)=d(y,x)$ for all $x,y\in X$;

  $-$ \,  $d(x,x)=0$ for all $x\in X$;

  $-$ \, $d(x,y)=k$ for all  $x, y\in X$ with $x\not \subset y$ (i.e., $x$ not a  face of~$y$) and
$y\not \subset x$;

  $-$ \, $d(x,y)=u$ for any 1-simplex $y\in X$ and any vertex~$x$ of~$y$; 

  $-$ \, $d(x,y)=v$ for any 2-simplex $y\in X$ and any edge~$x$ of~$y$;

  $-$ \, $d(x,y)=u+v$ for any 2-simplex $y\in X$ and  any vertex~$x$ of~$y$.

The last three conditions imply that the restriction of~$d$ to any subset of~$ X$ consisting of three simplexes $x\subset y \subset z $ of dimensions $0,1, 2$ respectively   is isometric to the  restriction of the standard metric in~$\Bbb R$ to the set $\{0,u,u+v\}\subset \Bbb R$.

The function~$d$ obviously satisfies all axioms of a metric except possibly the triangle inequality which needs to be verified. We check now that $d(x,y)+d(y,z)\geq d(x,z)$ for  any $x,y,z\in X$. If $x=y$ or $x=z$ or $y=z$, then this inequality holds because $d\geq 0$ and $d(x,x)=0$ for all~$x$.  Assume from now on that $x,y,z$  are three different simplexes of the triangulation~$\tau$. If  $x\subset y$ (i.e., if $x$ is a face of~$y$) and $y \subset z$ then necessarily $\dim(x)=0, \dim (y)=1, \dim (z)=2$ and 
 $$d(x,y)=u,  \,\, d(y,z)=v, \,\,d(x,z) =u+v.$$ In this case the triangle inequality  is an equality.
 The same is true if $x\supset y \supset z$. 
 
 We claim  that in all other cases $d(x,y)+d(y,z)> d(x,z)$.  
The symmetry of~$d$ ensures that exchanging~$x$ and~$z$  we get an equivalent inequality. Therefore it is enough to treat the case $\dim(x) \leq \dim(z)$ and  $x\neq y\neq z\neq x$.  By the definition of~$d$,  the numbers  $d(x,y),  d(y,z), d(x,z)$
belong to the set   $\{u,v,u+v, 
k\}$.   If~$x$ is not a face of~$z$, then $d(x,z)=k$. We  have
$d(x,y)+d(y,z)> k $ because in the set $\{u,v,u+v, 
k\}$ the sum of any two terms is strictly bigger than~$k$ (as  $u,v>k/2$). For the rest of the argument we assume that~$x$ is a face of~$z$.
We have several cases to consider.
 
 If~$x$ is a vertex of a 1-simplex~$z$, then the sum of any two terms 
 in the list $\{u,v,u+v, 
k\}$ being strictly bigger than~$k$   is also strictly bigger than $d(x,z)=u$. If~$x$ is an edge of a 2-simplex~$z$, then, similarly,  the sum of any two terms 
 in the set $\{u,v,u+v, 
k\}$ is strictly bigger than $d(x,z)=v$.
Suppose now  that~$x$ is a vertex of  a 2-simplex~$z$. The case $x\subset y\subset z$ being already considered, we  suppose that $x\not\subset y$ or $y\not \subset z$ or both.
If $x\not\subset y\subset z$, then either~$y$ is a vertex of~$z$ different from~$x$ or~$y$ is an edge of~$z$ opposite to~$x$. In the first case
$$d(x,y)+d(y,z)=k+u+v>u+v=d(x,z).$$ In the second case $$d(x,y)+d(y,z)=k+v>u+v=d(x,z).$$
If $x\subset y\not \subset z$, then  either $\dim(y)=1$ and then $$d(x,y)+d(y,z)=u+k>u+v=d(x,z)$$ or $\dim (y)=2$ and then $$d(x,y)+d(y,z)=u+v+k>u+v=d(x,z).$$ 
 If $x\not\subset y\not \subset z$, then $$d(x,y)+d(y,z)=2k>u+v=d(x,z).$$

We conclude that $(X, d)$ is a metric space. Moreover, the arguments above show  that the only strings of length~3 in~$X$ are the triples consisting of a 2-simplex of~$\tau$, an edge of this 2-simplex, and a vertex of this edge. The direct order in such a string holds the edge  as the middle term of the triple. This easily  implies that  there are no strings of length $>3$ in~$X$.  The triples  as above bijectively correspond to the 2-simplexes of the barycentric subdivision of~$\tau$. The string space $\vert X,d\vert$ is obtained by gluing  the 2-simplexes  numerated by those triples. These gluings are the same as the gluings of the 2-simplexes in the barycentric subdivision of~$\tau$. Consequently, the space  $\vert X,d\vert$ is homeomorphic to the  space~$T$. Under this homeomorphism the natural triangulation of $\vert X, d \vert$ corresponds to the barycentric subdivision of~$\tau$.  \end{proof} 

 \begin{corol}\label{cthm1}  Every compact surface (possibly, non-connected and/or with boundary)  is homeomorphic to the string space of a finite metric space. \end{corol}
 
 It is interesting to compute for every compact surface~$T$ the minimal number $n(T)$  of points in a metric space whose string space  is homeomorphic to~$T$. The proof of Theorem~\ref{thm1} shows that $n(T)$ is smaller than or equal to the minimal total number of simplexes in a triangulation of~$T$. This estimate seems to be rather weak. For example, the minimal triangulation of the 2-sphere $S^2$
 has  $14$ simplexes:~4 of  dimension zero,~6 of dimension one, and~4 of dimension~ two. At the same time,  Example~\ref{exa}.5 shows that $n(S^2)\leq 4$. It is straightforward to prove that $n(S^2)=4$.
 
 \section{$\varepsilon$-strings and $\varepsilon$-string spaces}
  
 
In this section we let~$\varepsilon$ be any non-negative real number. In generalization of strings and string spaces we introduce $\varepsilon$-strings and $\varepsilon$-string spaces.
  
  \subsection{$\varepsilon$-strings} For~$n\geq 3$, an  \emph{$\varepsilon$-string of  length~$n$}  in a gap space $(X,d)$ is an $n$-element  set $S \subset X$ whose elements   can be ordered  $x_1, x_2, \ldots, x_{n}$ so that   
 \begin{equation}\label{wo541}  d(x_1, x_{n})+\varepsilon \geq d(x_1, x_{2})+d(x_2, x_{3})+\cdots+d(x_{n-1}, x_{n}) . \end{equation}
 Any order  in~$S$ satisfying  \eqref{wo541} is said to be \emph{direct}. An $\varepsilon$-string endowed with a direct order is  \emph{ordered}.
These notions extend to   1-element and 2-element subsets of~$X$ as in Section~\ref{strings}.
 For $x\in X$,  a 
 $1$-element set $\{x\} $ is  a \emph{$\varepsilon$-string of  length~$1$}  
if~$x$ belongs to  an $\varepsilon$-string of length~3. The only order in the set $\{x\} $ is direct.
A  $2$-element subset of~$ X$  is an \emph{$\varepsilon$-string of  length~$2$}  
if it is contained (as a subset) in an  $\varepsilon$-string of length~3.  An order in an $\varepsilon$-string of length~2 is  \emph{direct} if it extends to a direct order in an $\varepsilon$-string of length~3. 

Note that the right-hand side of Formula~\eqref{wo541}  is always greater than or equal to $ d(x_1, x_{n})$. Therefore for $\varepsilon=0$, this formula is equivalent to \eqref{wo35} and we recover the notion of a string from Section~\ref{strings}.

\begin{lemma}\label{le+141}  All non-void subsets of an $\varepsilon$-string are  $\varepsilon$-strings (called  $\varepsilon$-substrings). A direct order in an $\varepsilon$-string restricts to a direct order in each $\varepsilon$-substring. \end{lemma}
  
\begin{proof}  Let  $(X,d)$ be a gap space. Let $S\subset X$ be an $\varepsilon$-string of length~$n$ with direct order  $S=\{x_1,..., x_n\}$.   For $n=1,2,3$,  both claims of the lemma concerning the subsets of~$S$ follow directly from the definitions.   Assume   that $n\geq 4$. 
 We have  \begin{equation}\label{wo39} d(x_1, x_{n-1})+ d(x_{n-1},  x_n) +\varepsilon \geq d(x_1, x_{n})+\varepsilon \geq   \end{equation}
 $$ \geq d(x_1, x_{2})+d(x_2, x_{3})+\cdots+d(x_{n-2}, x_{n-1})+d(x_{n-1}, x_{n})
$$
 where we use  the triangle inequality and the definition of an $\varepsilon$-string.  Cancelling $d(x_{n-1},  x_n) $, we get
 $$d(x_1, x_{n-1})+ \varepsilon  \geq d(x_1, x_{2})+d(x_2, x_{3})+\cdots+d(x_{n-2}, x_{n-1}) .$$ Thus,  the set  $\{x_1,..., x_{n-1}\}$  is an $\varepsilon$-string with direct order  $x_1,..., x_{n-1}$.      Similarly,  $$d(x_1, x_{2})+d(x_2, x_{n})+ \varepsilon \geq d(x_1, x_{n})+ \varepsilon \geq$$
 $$\geq 
 d(x_1, x_{2})+d(x_2, x_{3})+\cdots+d(x_{n-1}, x_{n}). $$ Cancelling $d(x_1, x_2)$ we deduce that the  set $\{x_2,..., x_n\}$ is  an $\varepsilon$-string with direct order  $x_2,..., x_n$.     Now, pick any $j\in \{2,..., n-1\}$. We have 
 $$d(x_1, x_{n})+ \varepsilon \geq d(x_1, x_{2})+d(x_2, x_{3})+\cdots+d(x_{n-1}, x_{n})\geq $$
 $$ = \sum_{k=1}^{j-2} d(x_k, x_{k+1})+d(x_{j-1}, x_{j})
 +d( x_{j}, x_{j+1})+\sum_{k=j+1}^{n-1} d( x_{k}, x_{k+1}) \geq 
 $$
 $$\geq   \sum_{k=1}^{j-2} d(x_k, x_{k+1})+d(x_{j-1},  x_{j+1}) +\sum_{k=j+1}^{n-1} d( x_{k}, x_{k+1} ) $$
 where we  use the the definition of an $\varepsilon$-string and the triangle inequality. Therefore  the set   $\{x_1,...,x_{j-1}, x_{j+1},..., x_n\}$ is an $\varepsilon$-string with direct order  $x_1,...,x_{j-1}$, $ x_{j+1},..., x_n$.    This proves
 both claims of the lemma for any subset  of~$S$ having  $n-1$ elements. By induction,
 we obtain both claims for all non-void subsets of~$S$. 
\end{proof} 

  \subsection{The $\varepsilon$-string space}  By Lemma~\ref{le+141},  the collection of $\varepsilon$-strings in a gap space $(X,d)$ is an abstract simplicial complex. We  denote its geometric realization  by $\vert X,d \vert_\varepsilon$.  
This is a topological space glued from simplexes numerated by $\varepsilon$-strings in $(X,d)$. 
The simplexes corresponding to smaller $\varepsilon$-strings are faces of  the simplexes corresponding to bigger $\varepsilon$-strings. We call $\vert X,d\vert_\varepsilon$ the \emph{$\varepsilon$-string space} of  $(X,d)$.
 By the very definition, the space  $\vert X,d\vert_\varepsilon$ 
 admits a triangulation whose all vertexes and  edges are faces of certain 2-simplexes. For $\varepsilon=0$, we obtain the same space as in Section~\ref{com},  i.e., 
 $\vert X,d \vert_0=\vert X,d \vert$.
 
 For any real numbers  $\varepsilon'\geq \varepsilon\geq 0$, each $\varepsilon$-string in $(X,d)$ is also an $\varepsilon'$-string in $(X,d)$. This induces a simplicial embedding $i_{\varepsilon, \varepsilon'}:\vert X,d \vert_{\varepsilon}\hookrightarrow
 \vert X,d \vert_{\varepsilon'}$. Clearly, this construction is functorial, that is $i_{\varepsilon, \varepsilon}=\id$ and $i_{\varepsilon', \varepsilon''}\circ i_{\varepsilon, \varepsilon'}=i_{\varepsilon, \varepsilon''}$ for all $\varepsilon''\geq \varepsilon'\geq \varepsilon\geq 0$.
  In  terminology of topological data analysis (TDA),   we have got a filtered system of simplicial complexes $\{ \vert X,d \vert_\varepsilon\}_\varepsilon$. The standard methods of TDA produce  persistent homology and the barcode
of this filtered system.
 
  Given $\varepsilon\geq 0$, for any $\varepsilon'\geq \varepsilon$ sufficiently close to~$\varepsilon$, all $\varepsilon'$-strings in $(X,d)$ are $\varepsilon$-strings and the mapping  $i_{\varepsilon, \varepsilon'}$ is a homeomorphism. In particular, for $\varepsilon=0$ and sufficiently small $\varepsilon'$, we have a homeomorphism $i_{0, \varepsilon'}: \vert X,d \vert\approx \vert X,d \vert_{\varepsilon'}$. For sufficiently big~$\varepsilon$, all non-void subsets of~$X$ are $ \varepsilon$-strings, so  $ \vert X,d \vert_\varepsilon$ is a simplex with ${\rm{card}} (X)$ vertexes.

\end{document}